\documentclass[12pt,reqno]{amsart}
\usepackage{amsthm, amsfonts,  amsmath, amssymb}
\usepackage[all]{xy}
\usepackage{url}
\usepackage{comment}
\usepackage[a4paper, top = 1.2in, bottom = 1.2in, left = 1in, right = 1in]{geometry}
\UseComputerModernTips

\newtheorem{thm}{Theorem}[section]

\newtheorem{prop}[thm]{Proposition}
\theoremstyle{definition}
\newtheorem{dfn}[thm]{Definition}

\newtheorem{conj}[thm]{Conjecture}

\newtheorem{remark}[thm]{Remark}
\theoremstyle{plain}
\newtheorem{cor}[thm]{Corollary}
\numberwithin{equation}{section}

\newcommand{\C}{\mathbb{C}}

\newcommand{\N}{\mathbb{N}}
\newcommand{\Q}{\mathbb{Q}}

\newcommand{\R}{\mathbb{R}}

\newcommand{\Z}{\mathbb{Z}}

\newcommand{\mcO}{\mathcal{O}}

\newcommand{\lra}{\longrightarrow}
\newcommand{\ra}{\rightarrow}

\newcommand{\mrm}[1]{\mathrm{#1}}

\newcommand{\psmat}[4]{\bigl( \begin{smallmatrix} #1 & #2 \\ #3 & #4 \end{smallmatrix} \bigr)}

\title[Remarks on $q$-exponents of GMF's]{Remarks on $q$-exponents of generalized modular functions}
\author[N. Kumar]{Narasimha Kumar}
\email{narasimha.kumar@iith.ac.in}
\address{
Department of Mathematics \\
Indian Institute of Technology Hyderabad\\
Ordnance Factory Estate, Yeddumailaram - 502205\\
Andhra Pradesh, INDIA. 
}
\date{}
\begin{document}

\begin{abstract}
In this note, we prove multiplicity one theorems for generalized modular functions (GMF),
in terms of their $q$-exponents, and make a general statement about the nature of values that the prime $q$-exponents of a GMF can take.
We shall also study the integrality of general $q$-exponents of a GMF and give an upper bound on the first sign change of these $q$-exponents. 
\end{abstract}

\maketitle

In the theory of modular forms, multiplicity one theorems play a significant role.
In the integral weight case, there are several multiplicity one theorems in the literature, see~\cite{DW00},~\cite{Pia79} for more details. The case is similar for half-integral weight modular forms 
and Siegel modular forms, see~\cite{Koh92},~\cite{Sah13},~\cite{SW03}. However, for generalized modular functions, we have not found any such theorems
in the literature. In \S\ref{sign-section}, we prove several  multiplicity one theorems for generalized modular functions, in terms of their $q$-exponents.
In fact, we show that the number of sign changes of $q$-exponents itself determine the generalized modular
function, up to a non-zero scalar. We improve this version, 
by assuming the pair Sato-Tate conjecture (Conjecture~\ref{Sato-Tate-conj} in the text) for a pair of non-CM Hecke eigenforms of integral weight.

In \S\ref{natureofvalues}, we sharpen the result of~\cite[Thm. 1]{KM11}, which gives information about the nature of values that 
the prime $q$-exponents of a GMF can take.

 In \S\ref{first sign change}, we shall give an upper bound on the first sign change of general $q$-exponents
of generalized modular functions. In the integral (or half-integral) weight modular forms, there are several results 
on producing an explicit upper bound on the first sign change of their Fourier coefficients (cf.~\cite{CK02},~\cite{CK13},~\cite{Mat12}).
In the same section, we also show  that the $q$-exponents $c(n)(n \in \N)$ are non-zero and integral only for finitely many $n$'s,
which generalizes~\cite[Thm. 5.1]{Kum12} to general $q$-exponents.

\section{Preliminaries}	
\label{preliminaries} 
In this section, we will recall the definition of generalized modular functions and some basic results about them.
We refer the reader to~\cite{KM03} for more details.
\begin{dfn}
We say that $f$ is a generalized modular function (GMF) of integral weight $k$ on $\Gamma_0(N)$, if $f$ is a holomorphic 
function on the upper half-plane $\mathbb{H}$ and
$$ f \left(\frac{az+b}{cz+d} \right)= \chi(\gamma) (cz+d)^k f(z) \quad \forall \gamma=\psmat{a}{b}{c}{d} \in \Gamma_0(N)$$
for some (not necessarily unitary) character $\chi:\Gamma_0(N) \ra \C^{*}$.
\end{dfn}
We will also suppose that $\chi(\gamma) = 1$ for all parabolic $\gamma \in \Gamma_0(N)$ of trace $2$. 
We remark that in~\cite{KM03}, a GMF in the above sense was called as a parabolic GMF (PGMF).

Let $f$ be a non-zero generalized modular function of weight $k$ on $\Gamma_0(N)$. Then $f$
has an infinite product expansion 
\begin{equation}
\label{infinite-expansion}
 f(z) = c_0 q^h \prod_{n = 1}^{\infty} (1-q^n)^{c(n)},
\end{equation}
where the product on the right-hand side of~$\eqref{infinite-expansion}$ is convergent in a small neighborhood of $q=0$, where
$q = e^{2 \pi i z}$. Here $c_0$ is a non-zero constant, $h$ is the order of $f$ at infinity, and the $c(n) (n \in \N)$ are uniquely
determined complex numbers~\cite{BKO04},~\cite{ES96}. 

If $f$ is a GMF of weight $0$ on $\Gamma_0(N)$ with $\mrm{div}(f) =  0$,
then its logarithmic derivative
\begin{equation}
\label{log-derivative}
 g:= \frac{1}{2\pi i} \frac{f^{\prime}}{f}
\end{equation}
is a cusp form of weight $2$ on $\Gamma_0(N)$ with trivial character. Conversely, if one starts with such a cusp form $g$, then there
exists a GMF $f$ of weight $0$ on $\Gamma_0(N)$ with $\mrm{div}(f)=0$ such that~\eqref{log-derivative} is satisfied and $f$ is uniquely determined 
up to a non-zero scalar \cite{KM03}. Suppose that the Fourier expansion of $g(z)$ is given by
\begin{equation*}
g(z) = \sum_{n = 1}^{\infty} b(n) q^n.
\end{equation*}
Let $K_f$ (resp., $K_g$) be the field generated by the $q$-exponents $c(n)$ (resp.,  $b(n)$) of $f$
(resp., of $g$) over $\Q$. Then, $K_f=K_g$, since for $n \geq 1$, 
\begin{eqnarray}
\label{key-equation}
b(n) &= -{\sum}_{d|n} dc(d),\\
nc(n) &=- {\sum}_{d|n}\mu(d) b(n/d).
\end{eqnarray}

\section{Multiplicity one theorems}
\label{sign-section}
In this section, we state several  multiplicity one theorems for generalized modular functions,
in terms of their $q$-exponents. The basic idea of these proofs are motivated from the work of Inam and Wiese~\cite{IW}.

Now, we let us recall the Sato-Tate measure and the notion of natural density for subsets of $\mathbb{P}$,
where $\mathbb{P}$  denotes the set of all prime numbers.

\begin{dfn}
The Sato-Tate measure $\mu_{\mrm{ST}}$ is the probability measure on  $[-1,1]$ given by $\frac{2}{\pi} \sqrt{1-t^2} dt$. 
\end{dfn}
\begin{dfn}
Let $S$ be a subset of  $\mathbb{P}$. The set $S$ has  natural density $d(S)$
if the limit
\begin{equation}
\underset{x \ra \infty}{\mrm{lim}}\  \frac{\# \{ p \leq x: p\in S\}}{\# \{ p \leq x: p\in \mathbb{P}\}}
\end{equation}
exists and is equal to $d(S)$. 
\end{dfn}
\begin{remark}
$d(S)=0$ if $|S|<\infty$.
\end{remark}
The following theorem is a consequence of the Sato-Tate equidistribution theorem for non-CM Hecke eigenforms
and Theorem $2$ of~\cite{Mat12}.

\begin{thm}[Multiplicity one theorem-I]
\label{multiplication1thm}
Let $N_1, N_2 \geq 1$ be square-free integers. For $i=1,2$, let $f_i$ be a non-constant GMF of weight $0$ on $\Gamma_0(N_i)$ with $\mrm{div}(f_i)=0$
and $q$-exponents $\{c_i(n)\}_{n \in \N}$, resp., and suppose $g_i$'s (as in~\eqref{log-derivative}) are the corresponding
normalized Hecke eigenforms without CM.
If $c_1(p)$ and $c_2(p)$ have the same sign for every prime $p$ except those in a set $E$ with analytic density $\kappa \leq 6/25$,
then $N_1=N_2$ and $f_1= f_2$, up to a non-zero scalar. 
\end{thm}
\begin{proof}

Suppose that $g_i = \sum_{i=1}^{\infty} b_i(n)q^n$, for $i=1,2$.
 Since 
$  b_1(p)  =  1- pc_1(p) $, we have
\begin{equation*}
c_1(p)>0 \Longleftrightarrow  -1 \leq B_1(p) < \frac{1}{2 \sqrt{p}}, \quad c_1(p)<0 \Longleftrightarrow \frac{1}{2 \sqrt{p}} < B_1(p) \leq 1, 
\end{equation*}
where $B_1(p) = \frac{b_1(p)}{2\sqrt{p}}$, by definition. We see that,
if $c_1(p)$ (resp., $b_1(p)$) is positive, then it does not mean that $b_1(p)$  (resp., $c_1(p)$) is negative. So, the theorem is not an immediate
consequence of~\cite[Thm. 2]{Mat12}. 

However, we now show that, except possibly for a natural density zero set of primes,
the signs of $c_1(p)$ and $b_1(p)$ are exactly the opposite. To prove this, it is enough to show that
$$d\left(\left\lbrace p\ \mrm{prime}: p \nmid N, 0  \leq B_1(p) < \frac{1}{2 \sqrt{p}}\right\rbrace\right)=0.$$

For any fixed (but small) $\epsilon >0$, we have the following inclusion of sets 
$$ \left\lbrace p\leq x: p \nmid N,\  B_1(p)\in[0,\epsilon] \right\rbrace \supseteq 
\left\lbrace  p \leq x: p \nmid N,\ p> \frac{1}{4\epsilon^2},\ 0 \leq B_1(p) < \frac{1}{2\sqrt{p}} \right\rbrace.$$ 
Hence, we have
$$ \# \{p\leq x: p \nmid N,\  B_1(p)\in[0,\epsilon]\} + \pi\left(\frac{1}{4\epsilon^2}\right) \geq \# \{ p \leq x:  
p \nmid N,\  0 \leq B_1(p) < \frac{1}{2\sqrt{p}}\}.$$
Now divide the above inequality by $\pi(x)$
$$ \frac{\# \{p\leq x:  p \nmid N,\  B_1(p)\in[0,\epsilon]\}}{\pi(x)} + \frac{\pi\left(\frac{1}{4\epsilon^2}\right) }{\pi(x)} 
\geq \frac{\# \left\lbrace p \leq x: p \nmid N,\   0 \leq B_1(p) < \frac{1}{2\sqrt{p}} \right\rbrace}{\pi(x)}.$$
The term $\frac{\pi(\frac{1}{4\epsilon^2})}{\pi(x)}$ tends to zero as $x\ra \infty$ as $\pi(\frac{1}{4\epsilon^2})$ is finite.
By the Sato-Tate equidistribution theorem (\cite[Thm. B.]{BGHT11}), we have
$$ \frac{\# \{ p\leq x : p \nmid N,\ B_1(p)\in[0,\epsilon]    \} }{\pi(x)}  \lra \mu_{\mrm{ST}}([0,\epsilon]) \quad\ \mrm{as}\quad 	x \ra \infty. $$ 
This implies that 
\begin{equation}
\label{key-inequality-1}
\underset{x \ra \infty}{\mrm{lim\ sup}} \ \frac{\{p\leq x: p \nmid N, 0  \leq B_1(p) < \frac{1}{2 \sqrt{p}} \}}{\pi(x)} \leq \mu_{\mrm{ST}}([0,\epsilon ]).
\end{equation}
Since the inequality~\eqref{key-inequality-1} holds for all $\epsilon>0$, we have that 
$$ \underset{x \ra \infty}{\mrm{lim}} \ \frac{\{p\leq x : p \nmid N, 0  \leq B_1(p) < \frac{1}{2 \sqrt{p}} \}}{\pi(x)} = 0.$$

Now, if $c_1(p)$ and $c_2(p)$ have same sign for every prime $p$ except those in a set $E$ with analytic density $\kappa \leq 6/25$,
then $b_1(p)$ and $b_2(p)$ also have same sign for every prime $p \not \in E$, since the signs of $b_1(p)$ and $c_1(p)$ are 
exactly the opposite except possibly for a  natural density zero set of primes. By~\cite[Thm. 2]{Mat12}, we have that $N_1=N_2$ and $g_1=g_2$.
Since $g_i$'s determine $f_i$, up to a non-zero scalar, we have that $f_1=f_2$, up to a non-zero scalar.
\end{proof}

In particular, we have:
\begin{cor}
Let $f_1,f_2,g_1,g_2$ be as in the above Theorem. Suppose that $c_1(p)$ and $c_2(p)$ have the same sign for every prime $p$, then 
$f_1$ is equal to $f_2$, up to a non-zero scalar, i.e., signs of the prime $q$-exponents determine the GMF, up to a non-zero scalar.
\end{cor}

Now, we shall state a stronger version of the multiplicity one theorem for GMF's, by assuming the pair Sato-Tate conjecture for 
non-CM Hecke eigenforms of integral weight (Conjecture~\ref{Sato-Tate-conj}).  
Now, we let us recall the pair Sato-Tate equidistribution conjecture. 

Let $g_1 = \sum_{n=1}^{\infty} b_1(n)q^n$ (resp., $g_2 = \sum_{n=1}^{\infty} b_2(n)q^n$) be a normalized cuspidal eigenform 
of weight $2k$ on $\Gamma_0(N)$. By Deligne's bound, for $i=1,2$, one has 
$$ |b_i(p)| \leq 2 p^{k-\frac{1}{2}}, $$
and we let 
\begin{equation}
\label{Sato-Tate-normalization}
B_i(p):=\frac{b_i(p)}{2p^{k-\frac{1}{2}}} \in [-1,1].
\end{equation}
We have the following pair Sato-Tate equidistribution conjecture for distinct cuspidal eigenforms $g_1,g_2$. 
\begin{conj}
\label{Sato-Tate-conj}
Let $k \geq 1$ and let $g_1,g_2$ be distinct normalized cuspidal eigenforms of weight $2k$ on $\Gamma_0(N)$
without CM. For any two subintervals $I_1, I_2 \subseteq [-1,1]$, we have 
$$ d(S(I_1,I_2)) = \underset{x \ra \infty}{\mathrm{lim}}  \ \frac{\#S(I_1,I_2)(x)}{\pi(x)} = \mu_{\mrm{ST}}(I_1) \mu_{\mrm{ST}}(I_2)= 
\frac{4}{\pi^2} \int_{I_1}  \sqrt{1-s^2}  ds  \int_{I_2}  \sqrt{1-t^2}  dt,$$
where $S(I_1,I_2) :=  \left\lbrace p \in \mathbb{P}: p\nmid N,B_1(p) \in I_1, B_2(p) \in I_2 \right\rbrace,
S(I_1,I_2)(x):=  \left\lbrace p\leq x: p\in S(I_1,I_2) \right\rbrace$.
In other words, the Fourier coefficients at primes are independently Sato-Tate distributed.
\end{conj}


For notational convenience, we let $\mathbb{P}_{<0}$ denote 
the set $\{ p \in \mathbb{P}: p\nmid N,\ c_1(p)c_2(p)<0 \}$, and similarly $\mathbb{P}_{>0}$, $\mathbb{P}_{\leq 0}$, $\mathbb{P}_{\geq 0}$, 
and $\mathbb{P}_{=0}$.  

\label{sign changes}
\begin{thm}[Non-CM case]
\label{main-thm-1}
For $i=1,2$, let $f_i$ be a non-constant GMF of weight $0$ on $\Gamma_0(N)$ with $\mrm{div}(f_i)=0$
and $q$-exponents $\{c_i(n)\}_{n \in \N}$, resp., and suppose $g_i$'s (as in~\eqref{log-derivative}) are the corresponding
normalized Hecke eigenforms without CM. If the pair Sato-Tate conjecture (Conj.~\ref{Sato-Tate-conj}) holds for the pair $(g_1,g_2)$,
then the product of $q$-exponents $c_1(p) c_2(p)$($p$ prime), change signs infinitely often.
Moreover, the sets $$\mathbb{P}_{>0}, \mathbb{P}_{<0}, \mathbb{P}_{\geq 0}, \mathbb{P}_{\leq 0}$$
have natural density $1/2$,
and $d(\mathbb{P}_{=0}) = 0$.
\end{thm}
We let $\pi_{<0}(x)$ denote the number 
$ \# \{p\leq x: p \in \mathbb{P}_{<0}\}$, and similarly for 
$\pi_{>0}(x), \pi_{\leq 0}(x)$, $\pi_{\geq 0}(x)$, and $\pi_{=0}(x)$. 

\begin{proof}[\textbf{Proof of Theorem~\ref{main-thm-1}}]
Since $b_i(p)  =  1- pc_i(p)$, we have
\begin{equation*}
c_i(p)>0 \Longleftrightarrow  -1 \leq B_i(p) < \frac{1}{2 \sqrt{p}}, \quad c_i(p)<0 \Longleftrightarrow \frac{1}{2 \sqrt{p}} < B_i(p) \leq 1, 
\end{equation*}
where $B_i(p) = \frac{b_i(p)}{2\sqrt{p}}$, by definition, for $i=1,2$. 
First, we shall show that
$$\underset{x \ra \infty}{\mrm{lim\ inf}} \ \frac{\pi_{< 0}(x)}{\pi(x)} \geq \mu_{\mrm{ST}}([0,1])=\frac{1}{2}.$$

For any fixed (but small) $\epsilon >0$, we have the following inclusion of sets
$$ \{p\leq x: p \nmid N,\  c_1(p)c_2(p)<0\} \supseteq S_{\frac{1}{4\epsilon^2}}([\epsilon,1],[-1,0])(x) \cup S_{\frac{1}{4\epsilon^2}}([-1,0],[\epsilon,1])(x),$$
where $S_a(I_1, I_2)(x):= \{p \in S(I_1,I_2)(x) : p>a \},$ for any $a \in \R^{+}$.
Hence, we have
$$ \# \{p\leq x: p \nmid N,\ c_1(p)c_2(p)<0\} + \pi \left(\frac{1}{4\epsilon^2}\right) \geq \# S([\epsilon,1],[-1,0])(x) + \# S([-1,0],[\epsilon,1])(x).$$
Now divide the above inequality by $\pi(x)$
$$ \frac{\# \{p\leq x:  p \nmid N,\ c_1(p)c_2(p)<0\}}{\pi(x)} + \frac{\pi\left(\frac{1}{4\epsilon^2}\right) }{\pi(x)} 
\geq \frac{\# S([\epsilon,1],[-1,0])(x) + \# S([-1,0],[\epsilon,1])(x)}{\pi(x)}.$$
The term $\frac{\pi(\frac{1}{4\epsilon^2})}{\pi(x)}$ tends to zero as $x\ra \infty$ as $\pi(\frac{1}{4\epsilon^2})$ is finite.
By Conjecture~\ref{Sato-Tate-conj}, we have
$$ \frac{\# S([\epsilon,1],[-1,0])(x) + \# S([-1,0],[\epsilon,1])(x) \} }{\pi(x)}  \lra 2. \mu_{\mrm{ST}}([\epsilon, 1])\mu_{\mrm{ST}}([-1,0]) \quad\ \mrm{as}\quad 	x \ra \infty. $$ 
This implies that 
\begin{equation}
\label{key-inequality}
\underset{x \ra \infty}{\mrm{lim\ inf}} \ \frac{\pi_{<0}(x)}{\pi(x)} \geq 2. \mu_{\mrm{ST}}([\epsilon, 1])\mu_{\mrm{ST}}([-1,0])=\mu_{\mrm{ST}}([\epsilon, 1]),
\end{equation}
where $\pi_{<0}(x)=\# \{p\leq x:p \nmid N,\ c_1(p)c_2(p)<0\}$ by definition. Since the inequality~\eqref{key-inequality} holds for all $\epsilon>0$, we have that 
$$ \underset{x \ra \infty}{\mrm{lim\  inf}} \ \frac{\pi_{<0}(x)}{\pi(x)} \geq \mu_{\mrm{ST}}([0,1]) = \frac{1}{2}.$$
A similarly proof  shows that
$\underset{x \ra \infty}{\mrm{lim\ inf}} \ \frac{\pi_{\leq 0}(x)}{\pi(x)} \geq \frac{1}{2}.$
Since $\pi_{> 0}(x) = \pi(x) - \pi_{\leq 0}(x)$, we have that
$\underset{x \ra \infty}{\mrm{lim\ sup}} \ \frac{\pi_{> 0}(x)}{\pi(x)} \leq \frac{1}{2}.$
Hence, the limit  $ \underset{x \ra \infty}{\mrm{lim}} \frac{\pi_{<0}(x)}{\pi(x)}$ exists 
and is equal to $\frac{1}{2}$. Therefore, the natural density of the set $\mathbb{P}_{<0}$ is $\frac{1}{2}$. 

Similarly, one can also argue for the sets $\mathbb{P}_{>0}$, $\mathbb{P}_{ \leq 0}$, and $\mathbb{P}_{ \geq 0}$, 
and show that the natural densities of these sets are $\frac{1}{2}$. The claim for $\mathbb{P}_{=0}$ follows from 
the former statements.
\end{proof}
Finally, we state another, but a stronger, version of Theorem~\ref{multiplication1thm},
assuming the Conjecture~\ref{Sato-Tate-conj}. Observe that, in this version, we dont require that $N$ to be square-free.
\begin{thm}[Multiplicity one theorem-II]
\label{multiplication1thm-1}
For $i=1,2$, let $f_i$ be a non-constant GMF of weight $0$ on $\Gamma_0(N)$ with $\mrm{div}(f_i)=0$
and $q$-exponents $\{c_i(n)\}_{n \in \N}$, resp., and suppose $g_i$'s (as in~\eqref{log-derivative}) are the corresponding
normalized Hecke eigenforms without CM. Suppose that the Conjecture~\ref{Sato-Tate-conj} holds for the pair $(g_1,g_2)$.
If $c_1(p)$ and $c_2(p)$ have the same sign for every prime $p$ except those in a set $E$ with analytic density $\kappa < 1/2$,
then $f_1= f_2$, up to a non-zero scalar. 
\end{thm}

\section{Nature of values taken by $c(p)$($p\ \mrm{prime}$)}
\label{natureofvalues}
In this section, we shall also make a general statement about the nature of values that the prime $q$-exponents of a GMF can take.
Kohnen and Meher showed that $q$-exponents $c(p)$($p$\ prime) take infinitely many values.
More precisely, they showed that
\begin{thm}[\cite{KM11}]
\label{main-thm}
Let $f$ be a non-constant GMF of weight $0$ on $\Gamma_0(N)$ with $div(f) = 0$, $q$-exponents $c(n)(n\in \N)$. 
Suppose that $g$ $($as in~\eqref{log-derivative}$)$ is a normalized Hecke eigenform.
Then the $c(p)(p\ \mrm{prime})$ take infinitely many different values. 
\end{thm}

In the above result, they did not mention anything about the nature of values that these exponents can take. 
However, using the recent results of~\cite{Kum12}, one can have a precise information about the nature
of values that $c(p)$($p$ prime) can take.  


\begin{prop}
\label{main-prop}
Suppose $f,g$ be as in Theorem~\ref{main-thm}. 
Then $c(p)(p\  \mrm{prime})$ take 
infinitely many (distinct) positive values. Moreover, almost all these positive values are non-integral.
\end{prop}
\begin{proof}
Since the character of $g$ is trivial, we see that $K_g$ is totally real, in particular $K_g \subset \R$. 
Hence, all $c(p)(p\  \mrm{prime})$ are real numbers, because $K_f=K_g$. If $g = \sum_{n=1}^{\infty} b(n)q^n$, with $b(1)=1$, has CM,
then $b(p)=0$ for infinitely many primes $p$. Hence, $c(p) =\frac{1}{p}$ for those infinitely many primes $p$ and this proves the proposition.
Hence, WLOG, let us assume that $g$ is without CM.

Now, suppose that the proposition is not true.  Let $a_1, \ldots, a_r$ be the only positive real numbers taken by $c(p)(p\  \mrm{prime})$.
Observe that
$$\{ p \in \mathbb{P} |\ c(p)>0 \} = \cup_{i=1}^{r} S_{a_i}, $$
where $S_{a_i} = \{ p \in \mathbb{P} |\ c(p)= a_i \}.$ 
It follows that the set $S_{a_i}$ is finite, since $$ \frac{1}{p}-\frac{2}{\sqrt{p}} \leq c(p)=a_i \leq     \frac{1}{p}+\frac{2}{\sqrt{p}}.$$ 
This implies that the natural density of $\{ p \in \mathbb{P} |\ p \nmid N,  c(p)>0 \})$ is zero, which 
is a contradiction, by~\cite[Thm. 2.6]{Kum12}. Moreover, all these values have to be non-integral, 
i.e., they don't belong to $\mcO_{K_f}$, by~\cite[Thm. 5.1]{Kum12}.
\end{proof}
Similarly, one can also show that $c(p)(p\  \mrm{prime})$ take infinitely many (distinct) negative values,
which are almost all non-integral, if $g$ does not have CM.

\section{Remarks on $q$-exponents $c(n)(n \in \N)$ of GMF's}
\label{first sign change}
In this final section, we make some remarks about the general $q$-exponents of GMF's. Especially, 
we shall study the integrality of general $q$-exponents of a GMF and give an upper bound on the first sign change of these $q$-exponents. 

The following proposition can be thought of as a generalization of Theorem $5.1$ in~\cite{Kum12}.
Now, we show:
\begin{prop}
\label{corollary-simple}
Let $f$ be a non-constant GMF of weight $0$ on $\Gamma_0(N)$ with $div(f) = 0$ with rational $q$-exponents $c(n)(n \in \N)$.
Suppose that $g$ $($as in~\eqref{log-derivative}$)$ is a normalized Hecke eigenform. Then, $c(n)$'s 
are non-zero and integral for only finitely many $n$, i.e., there exists a $N_0(f) \in \N$ such that 
$c(n)$ is non-integral, if non-zero, for all $n \geq N_0(f)$.
\end{prop}
\begin{proof}
By~\eqref{key-equation}, we know that $nc(n)$'s are integral for all $n \in \N$.  
However, it may not be  that $c(n)$'s are itself integral. We confirm this by showing that
$c(n)$'s are integral only finitely many times, unless they are zero.

Suppose that the proposition is not true. Let $S$ be an infinite subset of $\N$ such that $c(n) \in \Z-\{0\}$, for all $n \in S$.
Recall that $g(z) = \sum_{n=1}^{\infty} b(n) q^n$ with $b(1)=1$. 
Therefore, $n$ divides $b(n)+\sum_{d \mid n, d<n} dc(d)$ in $\Z$, for all $n$ in $S$.
Now, we show that this cannot happen.  

Since $|b(n)| \leq \sigma_0(n) \sqrt{n}$ and $nc(n)=-\sum_{d \mid n} \mu(d) b(n/d)$, we have
\begin{equation}
|nc(n)| \leq  \sum_{d|n} |b(n/d)| \leq \sum_{d|n} \sigma_0(n/d) \sqrt{n/d}  
          \leq \sqrt{n} \sum_{d|n} \sigma_0(n/d)  \leq  \sqrt{n} \sigma_0(n)^2.
\end{equation}

If $c(n) \in \Z-\{0\}$, then 
$n$ divides $b(n)+\sum_{d \mid n, d<n} dc(d)$ in $\Z$, then 
$$ n \leq |b(n)+\sum_{d \mid n, d<n} dc(d)| \leq \sigma_0(n)\sqrt{n} + \sum_{d \mid n, d<n} \sqrt{d} \sigma_0(d)^2 \leq 2 \sqrt{n} \sigma_0(n)^3.$$
This implies that
\begin{equation}
\label{integral-exponents}
 n \leq 2\sqrt{n} \sigma_0(n)^3.
\end{equation}
Since $\sigma_0(n) = o(n^{\epsilon})$ for all $\epsilon >0$,
the above inequality can only hold for finitely many $n$'s, hence $S$ is a finite set. Therefore, $c(n)$ is non-zero and integral for only finitely many $n$'s.
\end{proof}

The following corollary, which gives an another proof of~\cite[Thm. 1]{KM08-N}.
\begin{cor}[Kohnen-Mason]
\label{exponent-thm-3}
Let $f =\sum_{n=0}^{\infty} a(n) q^n$ be a GMF of weight $0$ and level $N$ with $div(f) = 0$.
Suppose $g$ $($as in~\eqref{log-derivative}$)$ is a normalized Hecke eigenform. 
If $a(0)=1$ and $a(n)\in \Z$ for $n \in \N$, then $f = 1$ is constant.
\end{cor}
\begin{proof}
Assume that $f \neq 1$.
Since $K_f=\Q$, one can show that if $a(0)=1$ and $a(n)\in \Z$ for all $n \in \N$, 
then $c(n) \in \Z$ for all $n \in \N$ (cf. the proof of~\cite[Thm. 1]{KM08-N}).

By~\cite[Theorems $2.6$, $2.7$]{Kum12}, we see that $c(n)$ cannot be equal to $0$,
for all $n \gg 0$. Therefore, by proposition~\ref{corollary-simple},
there exists $n \in \N$ such that $c(n)$ is non-integral,
which is a contradiction.  
\end{proof}
\begin{remark}
In the above proof, we could have also used Proposition~\ref{main-prop} to get the contradiction.
\end{remark}

Now, we give an upper bound on the first sign change of general $q$-exponents of a GMF.
Similar results have been considered in~\cite{CK13} for half-integral weight modular forms,
in~\cite{CK02} for integral weight modular forms. 
\begin{prop}
Let $N \in \N$ be square-free integer and $f$ be a non-constant GMF of weight $0$ on $\Gamma_0(N)$ with $\mrm{div}(f) = 0$ such that
its $q$-exponents $c(n)(n \in \N)$ are real. Then there exists $d_1, d_2 \in \N$ with
$$d_1,d_2 \ll N_0: = N^5 \log^{10}(N) \mrm{exp}\left(c\frac{\log(N+1)}{\log\log(N+2)}\right)\mrm{max}\{\psi_2(N), 4\sqrt{N}\log^{16}(2N)\} $$
such that $c(d_1)>0, c(d_2)<0.$ Here, $c>0$ is an absolute constant and $\psi_2(N):= \prod_{p \mid N}\frac{\log(2N)}{\log p}$. 
\end{prop}
\begin{proof}
Let $g$ be the logarithmic derivative of $f$ as in~\eqref{log-derivative}. 
By~\cite{KM03}, we know that $g=\sum_{n=1}^{\infty} b(n)q^n \in S_2(\Gamma_0(N))$.
The Fourier coefficients of $g$ and $q$-exponents of $f$ are related by:
\begin{equation}
\label{key-local}
b(n) = -{\sum}_{d|n} dc(d)
\end{equation}
The main theorem of~\cite{CK02} implies that there exists $n_1, n_2 \in \N$ with $n_1,n_2 \ll N_0$ such that 
such that $b(n_1)<0, b(n_2)>0$. Hence, there exists a divisor $d_1$ (resp., $d_2$) of $n_1$ (resp.,  of $n_2$)
such that $c(d_1)>0$ (resp., $c(d_2)<0$), by~\eqref{key-local}.
Since, $d_i \leq n_i$  and $n_i  \ll N_0$, we have that $d_i \ll N_0$, for $i=1,2$.
\end{proof}
When the logarithmic derivative of $f$ (as in~\eqref{log-derivative}) is a normalized Hecke eigenform,
the above bound can be improved.

\begin{prop}
Let $f$ be a non-constant GMF of weight $0$ on $\Gamma_0(N)$ with $\mrm{div}(f) = 0$ and  with $q$-exponents $c(n)(n \in \N)$.
Suppose $g$ is (as in~\ref{log-derivative}) a normalized Hecke eigenform. Then there exists 
$1< d_0$ with $d_0 \ll (4N)^{\frac{3}{8}}$ and $(d_0,N)=1$
such that $c(d_0)>0$, where the implied constant is absolute.
\end{prop}
\begin{proof}
By~\cite{KM03}, we know that $g=\sum_{n=1}^{\infty} b(n)q^n \in S_2(\Gamma_0(N))$.
The Fourier coefficients of $g$ and $q$-exponents of $f$ are related by:
\begin{equation}
\label{key-local-1}
b(n) = -{\sum}_{d|n} dc(d).
\end{equation}
Observe $c(1)=-b(1)=-1 < 0$. By~\cite[Thm. 1]{Mat12}, there exists $n_0 \in \N$ with $n_0 \ll (4N)^{\frac{3}{8}}$ such that 
$b(n_0)<0$. Hence, there exists a divisor $d_0\ (>1)$ of $n_0$ such that $c(d_0)>0$,
by~\eqref{key-local-1}. Since, $d_0 \leq n_0$  and $n_0  \ll (4N)^{\frac{3}{8}}$, we have that $d_0 \ll (4N)^{\frac{3}{8}}$.
\end{proof}

\bibliographystyle{plain, abbrv}

\begin{thebibliography}{99ab}
\bibitem{BGHT11} Barnet-Lamb, Tom; Geraghty, David; Harris, Michael; Taylor, Richard. 
                 A family of Calabi-Yau varieties and potential automorphy II. 
                 Publ. Res. Inst. Math. Sci. 47 (2011), no. 1, 29--98.

\bibitem{BKO04} 
	Bruinier, Jan H.; Kohnen, Winfried; Ono, Ken. 
	The arithmetic of the values of modular functions and the divisors of modular forms. 
        Compos. Math. 140 (2004), no. 3, 552--566. 
        
\bibitem{CK02}        
        Choie, YoungJu; Kohnen, Winfried. 
        The first sign change of Fourier coefficients of cusp forms. 
        Amer. J. Math. 131 (2009), no. 2, 517--543.        
        
        
\bibitem{CK13}        
        Choie, YoungJu; Kohnen, Winfried. 
        A short remark on sign changes of Fourier coefficients of half-integral weight modular forms.
        Arch. Math. 100 (2013), 333--336
        
\bibitem{DW00} 
        Duke, W.; Kowalski, E. 
        A problem of Linnik for elliptic curves and mean-value estimates for automorphic representations. 
        With an appendix by Dinakar Ramakrishnan. Invent. Math. 139 (2000), no. 1, 1--39.        
        



\bibitem{ES96} 
         Eholzer, Wolfgang; Skoruppa, Nils-Peter. 
         Product expansions of conformal characters. 
         Phys. Lett. B 388 (1996), no. 1, 82--89.
         
\bibitem{IW}   Inam, Ilker; Wiese, Gabor.
               Equidistribution of signs for modular eigenforms of half-integral weight.
	       \url{http://arxiv.org/abs/1210.2319} 
         



\bibitem{IKS07}
        Iwaniec, Henryk; Kohnen, Winfried; Sengupta, Jyoti 
        The first negative Hecke eigenvalue. 
        Int. J. Number Theory 3 (2007), no. 3, 355--363.

\bibitem{KM03}
        Knopp, Marvin; Mason, Geoffrey. 
        Generalized modular forms. 
        J. Number Theory 99 (2003), no. 1, 1--28.

\bibitem{Koh92} 
        Kohnen, Winfried. 
        On Hecke eigenforms of half-integral weight. 
        Math. Ann. 293 (1992), no. 3, 427--431.
 
        
        
        
        
        
\bibitem{KM08-N}
        Kohnen, Winfried; Mason, Geoffrey. 
        On generalized modular forms and their applications. 
        Nagoya Math. J. 192 (2008), 119--136.

\bibitem{KM08}
         Kohnen, Winfried; Martin, Yves.  
         On product expansions of generalized modular forms. 
         Ramanujan J. 15 (2008), no. 1, 103--107.

\bibitem{KM11}
        Kohnen, Winfried; Meher, Jaban. 
        Some remarks on the $q$-exponents of generalized modular functions. 
        Ramanujan J. 25 (2011), no. 1, 115--119.

        
\bibitem{Kum12}
        Kumar, Narasimha.
        On sign changes of $q$-exponents of generalized modular functions.
        J. Number Theory 133 (2013), no. 11, 3589--3597.

\bibitem{Pia79}        
        Piatetski-Shapiro, I. 
        Multiplicity one theorems. 
        Automorphic forms, representations and $L$-functions, 
        Part 1, pp. 209--212, Proc. Sympos. Pure Math., XXXIII, Amer. Math. Soc., Providence, R.I., 1979.
        
\bibitem{Mat12}
        Matom\"aki, Kaisa.
        On signs of Fourier coefficients of cusp forms.
        Mathematical Proceedings of the Cambridge Philosophical Society 152 (2012),
        no. 02, 207--222.
        
        
\bibitem{Sah13}
        Saha, Abhishek.
        Siegel cusp forms of degree 2 are determined by their fundamental Fourier coefficients. 
        Math. Ann.  355 (2013), no. 1, 363--380. 

\bibitem{SW03}
        Scharlau, Rudolf; Walling, Lynne H. 
        A weak multiplicity-one theorem for Siegel modular forms. 
        Pacific J. Math. 211 (2003), no. 2, 369--374.        
        
        
\end{thebibliography}

\end{document}